\documentclass{amsart}

\usepackage{amsmath}
\usepackage{amssymb}
\usepackage{epsfig}  		% For postscript
\usepackage{bbm}
\usepackage{hyperref}

\newtheorem*{tthm}{Theorem}

\newtheorem{thm}{Theorem}[section]
\newtheorem{prop}[thm]{Proposition}
\newtheorem{lem}[thm]{Lemma}
\newtheorem{cor}[thm]{Corollary}

\theoremstyle{definition}
\newtheorem{definition}[thm]{Definition}
\newtheorem{example}[thm]{Example}

\theoremstyle{remark}

\newtheorem{remark}[thm]{Remark}

\numberwithin{equation}{section}

\newcommand{\RR}{\mathbbm{R}}

\newcommand{\ZZ}{\mathbbm{Z}} 

\newcommand{\id}{\mathrm{id}}
\newcommand{\ad}{\mathrm{ad}}
\newcommand{\Ad}{\mathrm{Ad}}

\newcommand{\Iso}{\mathrm{Iso}}

\newcommand{\hol}{\mathrm{Hol}}

\DeclareMathOperator{\im}{\mathrm{im}}
\DeclareMathOperator{\rk}{\mathrm{rk}}

\newcommand{\frg}{\mathfrak{g}}
\newcommand{\frh}{\mathfrak{h}}

\newcommand{\frz}{\mathfrak{z}}

\newcommand{\frt}{\mathfrak{t}}
\newcommand{\frv}{\mathfrak{v}}
\newcommand{\sso}{\mathfrak{so}}
\newcommand{\iso}{\mathfrak{iso}}

\newcommand{\ub}[2]{\underbrace{#1}_{#2}}

\begin{document}

%%
%% The title of the paper goes here.  Edit to your title.
%%

\title[Complete Flat Pseudo-Riemannian Homogeneous Spaces]{Holonomy Groups of Complete Flat Pseudo-Riemannian Homogeneous Spaces}

\author[Globke]{Wolfgang Globke}
\address{Department of Mathematics, 
Institute for Algebra and Geometry\\
Karlsruhe Institute of Technology, 76128 Karlsruhe, Germany}
\email{globke@math.uni-karlsruhe.de}

\begin{abstract} 
We show that
a complete flat pseudo-Riemannian homogeneous manifold
with non-abelian linear holonomy is of dimension $\geq 14$.
Due to an example constructed in a previous article \cite{BG}, this is a 
sharp bound.
Also, we give a structure theory for the fundamental groups of
complete flat pseudo-Riemannian manifolds in dimensions $\leq 6$.
Finally, we observe that every finitely generated torsion-free 2-step
nilpotent group can be realized as the fundamental group of a complete flat
pseudo-Riemannian manifold with abelian linear holonomy.
\end{abstract}

\maketitle
\tableofcontents
%%%%%%%%%%%%%%%%%%%%%%%%%%%%%%%%%%%%%%%%%%%%%%%%%%%%%%%%%%%%%%%%%%%%%%
\section{Introduction}
%%%%%%%%%%%%%%%%%%%%%%%%%%%%%%%%%%%%%%%%%%%%%%%%%%%%%%%%%%%%%%%%%%%%%%

The study of pseudo-Riemannian homogeneous space forms
was pioneered by Joseph A. Wolf in the 1960s.
In the flat case, he proved that the
fundamental group $\Gamma$ of such a space $M$ is 2-step nilpotent.
For $M$ with abelian linear holonomy group he derived a representation
by unipotent affine transformations \cite{Wolf_1}.
The linear holonomy group $\mathrm{Hol}(\Gamma)$ of $M$ is the group 
consisting of the linear parts of $\Gamma$.
Wolf further proved that
for $\dim M\leq 4$ and Lorentz signatures, $\Gamma$ is a group
of pure translations, and that $\Gamma$ is free abelian for signatures
$(n-2,2)$.
It was unclear whether or not
non-abelian $\Gamma$ could exist for other signatures,
until Oliver Baues gave a first example in \cite{Baues}
of a compact flat pseudo-Riemannian homogeneous space with signature
$(3,3)$ having non-abelian fundamental group and abelian linear
holonomy group.

In an article \cite{BG} by Oliver Baues and the author, Wolf's unipotent
representations for fundamental groups with abelian $\mathrm{Hol}(\Gamma)$
were generalized for groups with non-abelian linear holonomy.
Also, it was shown that a (possibly incomplete) flat pseudo-Riemannian homogeneous
mani\-fold $M$
with non-abelian linear holonomy group is of dimension $\dim M\geq 8$.
In chapter \ref{chp_prelim} we review the main results about the
algebraic structure of the fundamental and holonomy groups of such $M$.

It was asserted in \cite{BG} that if $M$ is (geodesically) complete, then
$\dim M\geq 14$ holds.
This assertion is proved in chapter  \ref{chp_dimbound} of the present 
article.
More precisely, we prove the following:

\begin{tthm}\label{tthmA}
If $M$ is a complete flat homogeneous pseudo-Riemannian mani\-fold
such that its fundamental group $\Gamma$ has non-abelian linear holonomy group, then
\[
\dim M\geq 14
\]
and the signature $(n-s,s)$ of $M$ statisfies $n-s\geq s\geq 7$.
\end{tthm}

This estimate is sharp by an example given in \cite{BG}, which is repeated
in Example \ref{example2} for the reader's convenience.

In chapter \ref{chp_lowdim} we give a complete description of the
fundamental groups of
flat pseudo-Riemannian homogeneous spaces up to dimension $6$.
Although non-abelian fundamental groups may occur in dimension $6$, 
their holonomy groups are abelian as a consequence of the dimension bound in the
above theorem.

Further, we will see in chapter \ref{chp_fg} how any finitely generated torsion-free 2-step
nilpotent  group can be realized as the fundamental group of a complete
flat pseudo-Riemannian homogeneous manifold with abelian holonomy:

\begin{tthm}\label{tthmB}
Let $\Gamma$ be a finitely generated torsion-free 2-step nilpotent
group of rank $n$.
Then there exists a faithful representation
$\varrho:\Gamma\to\Iso(\RR^{2n}_n)$ such that
$M=\RR^{2n}_n/\varrho(\Gamma)$ a complete flat pseudo-Riemannian
homogeneous mani\-fold $M$ of signature $(n,n)$ with abelian linear
holonomy group.
\end{tthm}

%%%
\section{Preliminaries}
\label{chp_prelim}

Let $\RR^n_s$ be the space $\RR^n$ endowed with a non-degenerate symmetric
bilinear form of signature $(n-s,s)$ and $\Iso(\RR^n_s)$ its group of isometries. We assume $n-s\geq s$ throughout.
The number $s$ is called the \emph{Witt index}.
For a vector space $V$ endowed with a non-degenerate symmetric bilinear form
let $\mathrm{wi}(V)$ denote its Witt index.
Affine maps of $\RR^{n}$ are written 
as $\gamma= (I+A,v)$, where $I+A$ is the linear part ($I$ the
identity matrix), and $v$ the translation part.
Let $\im A$ denote the image of $A$.

Let $M$ denote a complete flat  pseudo-Riemannian 
homogeneous manifold.
Then $M$ is of the form $M=\RR^n_s/\Gamma$
with fundamental group $\Gamma\subset\Iso(\RR^n_s)$.
Homogeneity is determined by the
action of the
centralizer $\mathrm{Z}_{\Iso(\RR^n_s)}(\Gamma)$ of $\Gamma$ in
$\Iso(\RR^n_s)$ (see \cite[Theorem 2.4.17]{Wolf_buch}):

\begin{thm}\label{thm_transitive}
Let $p:\tilde{M}\to M$ be the universal pseudo-Riemannian covering of $M$
and let $\Gamma$ be the group of deck transformations.
Then $M$ is homogeneous if and only if
$\mathrm{Z}_{\Iso(\RR^n_s)}(\Gamma)$ acts transitively on $\tilde{M}$.
\end{thm}

This condition further implies that $\Gamma$ 
acts without fixed points on $\RR^n_s$. This constraint on
$\Gamma$ is the main difference to the more general case where $M$ is not required to be
geodesically complete.

Now assume $\Gamma\subset\Iso(\RR^n_s)$ has transitive centralizer in
$\Iso(\RR^n_s)$. 
We sum up some properties of $\Gamma$ for later reference (these are
originally due to \cite{Wolf_1}, see also \cite{globke}, \cite{Wolf_buch}
for reference).

\begin{lem}\label{lem_wolf1}
$\Gamma$ consists of affine transformations $\gamma=(I+A,v)$, where $A^2=0$,
$v\perp\im A$ and $\im A$ is totally isotropic.
\end{lem}

\begin{lem}\label{lem_wolf2}
For $\gamma_i=(I+A_i,v_i)\in\Gamma$, $i=1,2,3$, we have
$A_1 A_2 v_1=0=A_2A_1v_2$,
$A_1 A_2 A_3=0$
and
$[\gamma_1,\gamma_2] = (I+2A_1 A_2, 2A_1 v_2)$.
\end{lem}

\begin{lem}\label{lem_wolf3}
If $\gamma=(I+A,v)\in\Gamma$, then $\langle Ax,y\rangle=-\langle x,Ay\rangle$,
$\im A=(\ker A)^\bot$, $\ker A=(\im A)^\bot$ and $Av=0$.
\end{lem}

\begin{thm}\label{thm_wolf1}
$\Gamma$ is $2$-step nilpotent (meaning $[\Gamma,[\Gamma,\Gamma]]=\{\id\}$).
\end{thm}

For $\gamma=(I+A,v)\in\Gamma$, set $\hol(\gamma)=I+A$ (the linear component of
$\gamma$). We write $A=\log(\hol(\gamma))$.

\begin{definition}
The \emph{linear holonomy group} of $\Gamma$ is
$\hol(\Gamma)=\{\hol(\gamma) \mid \gamma\in \Gamma\}$.
\end{definition}

Let $x\in M$ and $\gamma\in\pi_1(M,x)$ be a loop. Then $\hol(\gamma)$
corresponds to the parallel transport $\tau_x(\gamma):\mathrm{T}_xM\rightarrow\mathrm{T}_xM$ in a natural way, see \cite[Lemma 3.4.4]{Wolf_buch}. This justifies the naming.

\begin{prop}\label{prop_wolf1}
The following are equivalent:
\begin{enumerate}
\item
$\hol(\Gamma)$ is abelian.
\item
If $(I+A_1,v_1), (I+A_2,v_2)\in\Gamma$, then $A_1 A_2=0$.
\item
The space $U_\Gamma=\sum_{\gamma\in\Gamma}\im A$ is totally isotropic.
\end{enumerate}
\end{prop}

Those $\Gamma$ with possibly non-abelian $\hol(\Gamma)$ were studied in \cite{BG}:
If $\hol(\Gamma)$ is not abelian, the space $U_\Gamma$ is not totally isotropic.
So we replace $U_\Gamma$ by the
totally isotropic subspace
\begin{equation}
U_0 = U_\Gamma \cap U_\Gamma^\bot = \sum_{\gamma\in\Gamma}\im A \cap \bigcap_{\gamma\in\Gamma} \ker A.
\label{eq_U0}
\end{equation}
We can find a \emph{Witt basis
for $\RR^n_s$ with respect to $U_0$}, that is a basis with the following properties:
If $k=\dim U_0$, there exists a basis for $\RR^n_s$,
\begin{equation}
\{ u_1,\ldots,u_k,\quad w_1,\ldots,w_{n-2k},\quad u_1^*,\ldots,u_k^* \},
\label{eq_nonabelian2}
\end{equation}
such that $\{u_1,\ldots,u_k\}$ is a basis of $U_0$,
$\{w_1,\ldots,w_{n-2k}\}$ is a basis of a non-degenerate subspace $W$ such that $U_0^\bot = U_0\oplus W$,
and $\{u_1^*,\ldots,u_k^*\}$ is a basis of a space $U_0^*$ such that
$\langle u_i,u_j^*\rangle=\delta_{ij}$ (then $U_0^*$ is called a
\emph{dual space} for $U_0$).
Then
\begin{equation}
\RR^n_s = U_0 \oplus W \oplus U_0^*
\label{eq_witt}
\end{equation}
is called a \emph{Witt decomposition} of $\RR^n_s$.
Let $\tilde{I}$
denote the signature matrix representing the restriction of
$\langle\cdot,\cdot\rangle$ to $W$ with respect to
the chosen basis of $W$.

In \cite[Theorem 4.4]{BG} we derived the following representation for
$\Gamma$:

\begin{thm}\label{thm_nonabelian0}
Let $\gamma=(I+A,v)\in\Gamma$ and fix a Witt basis with respect to $U_0$.
Then the matrix representation of $A$ in this basis is
\begin{equation}
A =
\begin{pmatrix}
0 & -B^\top\tilde{I} & C \\
0 & 0 & B \\
0 & 0 & 0
\end{pmatrix},
\label{eq_nonabelian3}
\end{equation}
with $B\in\RR^{(n-2k)\times k}$ and $C\in\sso_k$ (where $k=\dim U_0$).
The columns of $B$ are isotropic and mutually orthogonal with respect to
$\tilde{I}$.
\end{thm}

%%%%%%%%%%%%%%%%%%%%%%%%%%%%%%%%%%%%%%%%%%%%%%%%%%%%%%%%%%%%%%%%%%%%%%
\section{The Dimension Bound for Complete Manifolds}
\label{chp_dimbound}
%%%%%%%%%%%%%%%%%%%%%%%%%%%%%%%%%%%%%%%%%%%%%%%%%%%%%%%%%%%%%%%%%%%%%%

In this section we derive further properties of the matrix
representation in (\ref{eq_nonabelian3}).

%%%%%%%%%%%%%%%%%%%%%%%%%%%%%%%%%%%%%%%%%%%%%%%%%%%%%%%%%%%%%%%%%%%%%%
\subsection{Properties of the Matrix Representation}
\label{sec_matrep}
%%%%%%%%%%%%%%%%%%%%%%%%%%%%%%%%%%%%%%%%%%%%%%%%%%%%%%%%%%%%%%%%%%%%%%

We fix a Witt basis for $U_0$ as in the previous section.
Let $\gamma_i\in\Gamma$ with $\gamma_i=(I+A_i,v_i)$, $i=1,2$.
Then $B_i$ and $C_i$ refer
to the respective matrix blocks of $A_i$ in (\ref{eq_nonabelian3}).
Set $[\gamma_1,\gamma_2]=\gamma_3=(I+A_3,v_3)$.

\begin{lem}\label{lem_translation1}
We have $v_3=2A_1 v_2=-2A_2 v_1\in U_0$.
Further,
if $\gamma_3\neq I$ and $\Gamma$ acts freely, then $v_3\neq 0$.
\end{lem}
\begin{proof}
By Lemma \ref{lem_wolf2},
$v_3=2A_1 v_2=-2A_2 v_1\in\im A_1$.
Because $\Gamma$ is 2-step nilpotent, $\gamma_3$ is central.
Again by Lemma \ref{lem_wolf2}, $v_3\in\bigcap_{\gamma\in\Gamma}\ker A$.
Hence $v_3\in U_0$.

If $\Gamma$ acts freely and $\gamma_3\neq I$, then $v_3\neq 0$ because
otherwise $0$ would be a fixed point for $\gamma_3$.
\end{proof}

\begin{lem}\label{lem_translation2}
If $u_1^*, u_2^*$ denote the respective $U_0^*$-components of the translation
parts $v_1,v_2$, then
$u_1^*, u_2^* \in \ker B_1\cap\ker B_2$.
\end{lem}
\begin{proof}
Let $v_3=u_3+w_3+u_3^*$ be the Witt decomposition of $v_3$.
By Lemma \ref{lem_translation1}, $w_3=0,u_3^*=0$.
Writing out the equation $v_3=A_1 v_2=-A_2 v_1$ with (\ref{eq_nonabelian3})
it follows that $B_1 u_2^*=0=B_2 u_1^*$.
By Lemma \ref{lem_wolf3}, $B_1 u_1^*=0=B_2 u_2^*$.
\end{proof}

The following rules were already used in \cite[Theorem 5.1]{BG} to derive the general dimension bound for (possibly incomplete)
flat pseudo-Riemannian homogeneous manifolds:
\begin{enumerate}
\item
\emph{Isotropy rule:}
The columns of $B_i$ are isotropic and mutually orthogonal with respect to $\tilde{I}$ (Theorem \ref{thm_nonabelian0}).
\item
\emph{Crossover rule:}
Given $A_1$ and $A_2$,
let $b_2^i$ be a column of $B_2$ and $b_1^k$ a column
of $B_1$. Then
$\langle b_1^k,b_2^i\rangle = - \langle b_1^i,b_2^k\rangle$.
In particular, $\langle b_1^k, b_2^k\rangle=0$, and
$\langle b_1^i,b_1^k\rangle=0$. 
If $\langle b_1^i,b_2^k\rangle\neq 0$ then
$b_1^k,b_1^i,b_2^k,b_2^i$ are linearly independent.
(The product of $A_1 A_2$ contains $-B_1^\top\tilde{I}B_2$ as the skew-symmetric
upper right block, so its entries are the values $-\langle b_1^k,b_2^i\rangle$.)
\item
\emph{Duality rule:}
Assume $A_1$ is not central (that is $A_1A_2\neq0$ for some $A_2$).
Then $B_2$ contains a column $b_2^i$ and $B_1$ a column $b_1^j$ such that
$\langle b_1^j,b_2^i\rangle\neq 0$.
\end{enumerate}

\begin{lem}\label{lem_rk_BX}
Assume $A_1 A_2\neq 0$ and that the columns $b_1^i$ in $B_1$ and
$b_2^j$ in $B_2$ satisfy $\langle b_1^i, b_2^j\rangle\neq 0$.
The subspace $W$ in (\ref{eq_witt}) has a Witt decomposition
\begin{equation}
W = W_{ij}\oplus W'\oplus W_{ij}^*,
\label{eq_mini_witt}
\end{equation}
where $W_{ij}=\RR b_1^i\oplus\RR b_1^j$, $W_{ij}^*=\RR b_2^i\oplus\RR b_2^j$,
$W'\perp W_{ij}$, $W'\perp W_{ij}^*$,
and $\langle \cdot,\cdot\rangle$ is non-degenerate on $W'$.
Furthermore,
\begin{equation}
\mathrm{wi}(W)\geq\rk B_1\geq 2
\quad
\text{ and }\quad \dim W\geq 2 \rk B_1 \geq 4.
\label{eq_dimW}
\end{equation}
\end{lem}
\begin{proof}
$\RR b_1^i\oplus\RR b_1^j$ is totally isotropic because $\im B_1$ is.
By the crossover rule, $\{b_2^j, b_2^i\}$ is a dual basis to $\{b_1^i,b_1^j\}$
(after scaling, if necessary).

$W$ contains $\im B_1$ as a totally isotropic subspace, so it also
contains a dual space. Hence $\mathrm{wi}(W)\geq\rk B_1\geq\dim W_{ij}\geq 2$ and
$\dim W\geq2\rk B_1\geq2\dim W_{ij}=4$.
\end{proof}

%%%%%%%%%%%%%%%%%%%%%%%%%%%%%%%%%%%%%%%%%%%%%%%%%%%%%%%%%%%%%%%%%%%%%%
\subsection{Criteria for Fixed Points}
\label{sec_fixpt}
%%%%%%%%%%%%%%%%%%%%%%%%%%%%%%%%%%%%%%%%%%%%%%%%%%%%%%%%%%%%%%%%%%%%%%

In this subsection, assume the centralizer of $\Gamma\subset\Iso(\RR^n_s)$ has an open orbit
in $\RR^n_s$, but does not necessarily act transitively.

\begin{remark}
If the centralizer does act transitively on $\RR^n_s$, then $\Gamma$
acts freely: Assume $\gamma.p=p$ for some $\gamma\in\Gamma$, $p\in\RR^n_s$.
For every $q\in\RR^n_s$ there is $z\in\mathrm{Z}_{\Iso(\RR^n_s)}(\Gamma)$
such that $z.p=q$. So $\gamma.q=\gamma.(z.p)=z.(\gamma.p)=z.p=q$
for all $q\in\RR^n_s$. Hence $\gamma=I$.
\end{remark}

We will deduce some criteria for $\Gamma$ to have a fixed point,
which allows us to exclude such groups $\Gamma$ as fundamental groups
for \emph{complete} flat pseudo-Riemannian homogeneous manifolds.

Let $\Gamma,U_0,\gamma_1,\gamma_2,\gamma_3=[\gamma_1,\gamma_2],A_i,B_i,C_i$ be as in the
previous sections.
For any $v\in\RR^n_s$ let $v=u+w+u^*$ denote the Witt decomposition with
respect to $U_0$. From (\ref{eq_nonabelian3}) we get the following two
coordinate expressions which we use repeatedly:
\begin{equation}
A_1 A_2 =
\begin{pmatrix}
0 & -B_1^\top\tilde{I} & C_1\\
0 & 0 & B_1\\
0 & 0 & 0
\end{pmatrix}
\begin{pmatrix}
0 & -B_2^\top\tilde{I} & C_2\\
0 & 0 & B_2\\
0 & 0 & 0
\end{pmatrix}
=
\begin{pmatrix}
0 & 0 & -B_1^\top\tilde{I}B_2\\
0 & 0 & 0\\
0 & 0 & 0
\end{pmatrix}
\label{eq_matmat}
\end{equation}
(note that $A_3=2A_1 A_2$ as a consequence of Lemma \ref{lem_wolf2}),
and for $v\in\RR^n_s$
\begin{equation}
A_i v =
\begin{pmatrix}
0 & -B_i^\top\tilde{I} & C_i\\
0 & 0 & B_i\\
0 & 0 & 0
\end{pmatrix}
\begin{pmatrix}
u\\ w\\ u^*
\end{pmatrix}
=
\begin{pmatrix}
-B_i^\top\tilde{I} w + C_i u^*\\
B_i u^*\\
0
\end{pmatrix}.
\label{eq_matvec}
\end{equation}

In the following we assume that the linear parts of $\gamma_1,\gamma_2$
do not commute, that is $A_1 A_2\neq 0$.
In particular, $A_1 v_2\neq 0$.

\begin{lem}\label{lem_fp_rule}
If $u_3\in\im B_1^\top\tilde{I} B_2$, then $\gamma_3$ has a fixed point.
\end{lem}
\begin{proof}
We have $C_3=-B_1^\top\tilde{I}B_2$ by Lemma \ref{lem_wolf2} and
(\ref{eq_nonabelian3}).
By Lemma \ref{lem_translation1}, $v_3=u_3\in U_0$.
If there exists $u^*\in U_0^*$ such that $C_3 u^*=u_3$,
then $\gamma_3.(-u^*) = (I+A_3,v_3).(-u^*) = -u^* -C_3 u^* + u_3 = -u^*$.
So $-u^*$ is fixed by $\gamma_3$.
\end{proof}

\begin{lem}\label{lem_rk_ux_uy}
If $\rk B_1^\top\tilde{I}B_2=\rk B_1$ and the $\Gamma$-action is free,
then $u_1^*\neq 0, u_2^*\neq 0$.
\end{lem}
\begin{proof}
From (\ref{eq_matvec}) we get
$u_3 = -B_1^\top\tilde{I}w_2 + C_1 u_2^*$.
Also, $\im B_1^\top\tilde{I}B_2\subset\im B_1^\top$.
But by our rank assumption, $\im B_1^\top\tilde{I}B_2=\im B_1^\top$.

So, if $u_2^*=0$, then $u_3\in\im B_1^\top=\im B_1^\top\tilde{I}B_2$, which implies the
existence of a fixed point by Lemma \ref{lem_fp_rule}.
So $u_2^*\neq 0$ if the action is free.
Using $v_3=A_1 v_2=-A_2 v_1$, we can conclude $u_1^*\neq 0$
in a similar manner.
\end{proof}

\begin{cor}\label{cor_dimU0_2}
\label{cor_dimU0_2}
If $\dim U_0=2$, then the $\Gamma$-action is not free.
\end{cor}
\begin{proof}
By Lemma \ref{lem_rk_BX}, $2\leq\rk B_1\leq \dim U_0=2$, so $B_1$ is of full
rank. Now $A_1v_1=0$ and (\ref{eq_matvec}) imply $u_1^*=0$, so by
Lemma \ref{lem_rk_ux_uy}, the $\Gamma$-action is not free.
\end{proof}

\begin{lem}\label{lem_rk_dimU0_3}
If $\dim U_0=3$ and $\dim(\im B_1+\im B_2)\leq 5$, then
$\gamma_3$ has a fixed point.
\end{lem}
\begin{proof}
By Lemma \ref{lem_rk_BX}, $\rk B_1, \rk B_2\geq 2$.
We distinguish two cases:
\begin{enumerate}
\item[(i)]
Assume $\rk B_1=2$ (or $\rk B_2=2$).
Because $C_3=-B_1^\top\tilde{I}B_2\neq 0$ is skew, it is also of rank $2$.
Then
$\im B_1^\top\tilde{I}B_2=\im B_1^\top$.
$\ker B_1$ is a 1-dimensional subspace due to $\dim U_0^*=3$.
Because $u_1^*, u_2^*\in \ker B_1$, we have $u_1^*=\lambda u_2^*$ for some number $\lambda\neq 0$.

From (\ref{eq_matvec}) and $A_1 v_1=0$ we get
\begin{align*}
\lambda u_3 &= -B_1^\top\tilde{I}\lambda w_2 + C_1 \lambda u_2^*
=-B_1^\top\tilde{I}\lambda w_2 + C_1 u_1^*, \\
0&=-B_1^\top\tilde{I} w_1 + C_1 u_1^*.
\end{align*}
So $\lambda u_3 = \lambda u_3 -0
= B_1^\top\tilde{I}(w_1-\lambda w_2)$.
In other words, $u_3\in\im B_1^\top=\im B_1^\top\tilde{I}B_2$, and
$\gamma_3$ has a fixed point by Lemma \ref{lem_fp_rule}.
\item[(ii)]
Assume $\rk B_1=\rk B_2=3$.
As $[A_1,A_2]\neq 0$, the duality rule and the crossover rule imply
the existence of a pair of
columns $b_1^i$, $b_1^j$ in $B_1$ and a pair of columns $b_2^j$, $b_2^i$
in $B_2$ such that $\alpha=\langle b_1^i,b_2^j\rangle=-\langle b_1^j,b_2^i\rangle\neq 0$.
For simplicity say $i=1$, $j=2$.
As $\rk B_1=3$, the column $b_1^3$ is linearly independent of $b_1^1, b_1^2$,
and these columns span the totally isotropic subspace $\im B_1$ of $W$.
\begin{itemize}
\item
Assume $b_2^3\in\im B_1$ (or $b_1^3\in\im B_2$).
Then $b_2^3$ is a multiple of $b_1^3$:
In fact, let $b_2^3=\lambda_1 b_1^1+\lambda_2 b_1^2+\lambda_3 b_1^3$.
Then $\langle b_2^3,b_1^i\rangle=0$ because $\im B_1$ is totally isotropic.
Since $\im B_2$ is totally isotropic and by the crossover rule,
\begin{align*}
0 &= \langle b_2^3,b_2^1\rangle
=\lambda_1\langle b_1^1,b_2^1\rangle
+\lambda_2\langle b_1^2,b_2^1\rangle
+\lambda_3\langle b_1^3,b_2^1\rangle \\
&=\lambda_2\alpha - \lambda_3\langle b_2^3,b_1^1\rangle = \lambda_2\alpha.
\end{align*}
Because $\alpha\neq0$, this implies $\lambda_2=0$ and in the same way
$\lambda_1=0$. So $b_2^3=\lambda_3 b_1^3$.
Now $b_1^3\perp b_1^i, b_2^j$ for all $i,j$.
We have $u_2^*=0$ because $B_2 u_2^*=0$ and $B_2$ is of maximal rank.
Then $\langle b_1^3, w_2\rangle=\langle b_2^3,w_2\rangle=0$, because
$0=B_2^\top\tilde{I}w_2+C_2 u_2^*=B_2^\top\tilde{I} w_2$.
Hence $C_3$ and $u_3$ take the form
\[
C_3 = -B_1^\top\tilde{I} B_2
=
\begin{pmatrix}
0 & -\alpha & 0\\
\alpha & 0 & 0\\
0 & 0 & 0
\end{pmatrix},
u_3 = -B_1^\top\tilde{I} w_2
=
\begin{pmatrix}
-\langle b_1^1, w_2\rangle\\
-\langle b_1^2, w_2\rangle\\
0
\end{pmatrix}.
\]
It follows that $u_3\in\im C_3$, so in this case
$\gamma_3$ has a fixed point by Lemma \ref{lem_fp_rule}.
\item
Assume $b_2^3\not\in\im B_1$ and $b_1^3\not\in\im B_2$.
This means $b_2^3$ and $b_1^3$ are linearly independent.
If $b_2^3\perp\im B_1$, then $b_1^3\perp\im B_2$ by the crossover rule.
With respect to the Witt decomposition
$W=W_{12}\oplus W'\oplus W_{12}^*$ (Lemma \ref{lem_rk_BX}), this means
$b_1^3, b_2^3$ span a 2-dimensional subspace of
$(W_{12}\oplus W_{12}^*)^\bot=W'$. But then
$\dim(\im B_1+\im B_2)=6$, contradicting the lemma's assumption that this dimension
should be $\leq 5$.

So $b_2^3\not\perp\im B_1$ and $b_1^3\not\perp\im B_2$ hold.
Because further $b_1^3\perp\im B_1$, $b_2^3\perp\im B_2$ and
$\dim(\im B_1+\im B_2)\leq 5$, there exists a
$b\in W'$ (with $W'$ from the Witt decomposition above) such that
\begin{align*}
b_1^3 &= \lambda_1 b_1^1 + \lambda_2 b_1^2 + \lambda_3 b,\\
b_2^3 &= \mu_1 b_2^1 + \mu_2 b_2^2 + \mu_3 b.
\end{align*}
Because $B_1, B_2$ are of maximal rank, we have $u_1^*=0=u_2^*$ as a
consequence of $A_i v_i=0$.
Then
\[
0=B_2^\top\tilde{I}w_2=
\begin{pmatrix}
\langle b_2^1,w_2\rangle\\
\langle b_2^2,w_2\rangle\\
\langle b_2^3,w_2\rangle
\end{pmatrix},
\]
and this implies $\langle b,w_2\rangle=0$. Put
$\xi=\langle b_1^1,w_2\rangle$, $\eta=\langle b_1^2,w_2\rangle$. Then
\[
u_3 = -B_1^\top\tilde{I} w_2
=
-\begin{pmatrix}
\xi\\\eta\\ \lambda_1 \xi+\lambda_2 \eta
\end{pmatrix}
\]
and (recall $\alpha=\langle b_1^1,b_2^2\rangle=-\langle b_1^2,b_2^1\rangle$)
\[
C_3=B_2^\top\tilde{I}B_1
=
\begin{pmatrix}
0 & -\alpha & -\lambda_2\alpha\\
\alpha & 0 & \lambda_1\alpha\\
\lambda_2\alpha & -\lambda_1\alpha & 0
\end{pmatrix}.
\]
So
\[
C_3\cdot\frac{1}{\alpha}
\begin{pmatrix}
-\eta\\\xi\\0
\end{pmatrix}
=
-
\begin{pmatrix}
\xi\\\eta\\\lambda_1\xi+\lambda_2\eta
\end{pmatrix}=u_3.
\]
By Lemma \ref{lem_fp_rule}, $\gamma_3$ has a fixed point.
\end{itemize}
\end{enumerate}
So in any case $\gamma_3$ has a fixed point.
\end{proof}

\begin{lem}\label{lem_rk_dimU0_4}
If $\dim U_0=4$ and $\rk B_1^\top\tilde{I}B_2=\rk B_1=\rk B_2$,
then $\gamma_3$ has a fixed point.
\end{lem}
\begin{proof}
By assumption
\[
\im B_1^\top\tilde{I}B_2=\im B_1^\top=\im B_2^\top.
\]
First, assume $u_1^*=\lambda u_2^*$ for some number $\lambda\neq 0$.
Writing out $A_1 v_2=v_3$ and $A_1 v_1=0$ with (\ref{eq_matvec}), we get
\begin{align*}
\lambda u_3 &= -B_1^\top\tilde{I}\lambda w_2 + C_1 \lambda u_2^*
=-B_1^\top\tilde{I}\lambda w_2 + C_1 u_1^*, \\
0&=-B_1^\top\tilde{I} w_1 + C_1 u_1^*.
\end{align*}
So
\[
\lambda u_3 = \lambda u_3 -0
= B_1^\top\tilde{I}(w_1-\lambda w_2).
\]
In other words, $u_3\in\im B_1^\top=\im B_1^\top\tilde{I}B_2$, and
$\gamma_3$ has a fixed point by Lemma \ref{lem_fp_rule}.

Now, assume $u_1^*$ and $u_2^*$ are linearly independent.
Lemma \ref{lem_translation2} can be reformulated as
\[
\im B_1^\top=\im B_2^\top \subseteq \ker u_1^{*\top}\cap\ker u_2^{*\top}.
\]
$\ker u_1^{*\top}$, $\ker u_2^{*\top}$ are $3$-dimensional subspaces of
the $4$-dimensional space $U_0^*$, and their inter\-section is of dimension $2$
(because $u_1^*, u_2^*$ are linearly independent).
By Lemma \ref{lem_rk_BX}, $\rk B_1\geq 2$, so it follows that
\[
\im B_1^\top=\im B_2^\top = \ker u_1^{*\top}\cap\ker u_2^{*\top}.
\]
With $A_1 v_1=0$ and (\ref{eq_matvec}) we conclude
$C_1 u_1^* = b$ for some $b\in\im B_1^\top$. Thus, by the skew-symmetry of $C_1$,
\[
0= (u_2^{*\top} C_1 u_1^*)^\top = - u_1^{*\top} C_1 u_2^*.
\]
So $C_1 u_2^*\in\ker u_1^{*\top}$. In the same way $C_2 u_1^*\in\ker u_2^{*\top}$.
But $u_3=C_1 u_2^*+b_1=-C_2 u_1^* + b_2$ for some $b_1, b_2\in\im B_1^\top$. Hence
\begin{align*}
u_1^{*\top} u_3
&= \ub{u_1^{*\top}C_1 u_2^*}{=0} + \ub{u_1^{*\top} b_1}{=0} = 0,\\
u_2^{*\top} u_3
&= -\ub{u_2^{*\top}C_2 u_1^*}{=0} + \ub{u_2^{*\top} b_2}{=0} = 0. 
\end{align*}
So $u_3\in\ker u_1^{*\top}\cap\ker u_2^{*\top}=\im B_1^\top=\im B_1^\top\tilde{I}B_2$.
With Lemma \ref{lem_fp_rule} we conclude that there exists a fixed point
for $\gamma_3$.
\end{proof}

%%%%%%%%%%%%%%%%%%%%%%%%%%%%%%%%%%%%%%%%%%%%%%%%%%%%%%%%%%%%%%%%%%%%%%
\subsection{The Dimension Bound}
\label{sec_dimbound}
%%%%%%%%%%%%%%%%%%%%%%%%%%%%%%%%%%%%%%%%%%%%%%%%%%%%%%%%%%%%%%%%%%%%%%

Let $\Gamma, \gamma_1,\gamma_2,\gamma_3=[\gamma_1,\gamma_2]$ be as in the
previous subsection, let $\RR^n_s=U_0\oplus W\oplus U_0^*$ be
the Witt decomposition (\ref{eq_witt}), and let $A_i, B_i, C_i$ refer
to the matrix representation (\ref{eq_nonabelian3}) of $\gamma_i$.
We will assume that the linear parts $A_1,A_2$ of $\gamma_1,\gamma_2$ do
not commute, that is, $\hol(\Gamma)$ is not abelian.

\begin{thm}\label{thm_nonabelian1}
Let $\Gamma\subset\Iso(\RR^n_s)$ and assume the centralizer
$\mathrm{Z}_{\Iso(\RR^n_s)}(\Gamma)$ acts transitively on $\RR^n_s$.
If $\hol(\Gamma)$ is non-abelian, then
\[
s\geq 7\quad \text{ and }\quad n\geq 14.
\]
As Example \ref{example2} shows, this is a sharp lower bound.
\end{thm}
\begin{proof}
We will show $s\geq 7$, then it follows immediately from $n-s\geq s$ that
\[
n\geq 2s\geq 14.
\]
If the centralizer is transitive, then $\Gamma$ acts freely.
From Corollary \ref{cor_dimU0_2} we know that $\dim U_0\geq 3$.
By Lemma \ref{lem_rk_BX}, $\mathrm{wi}(W)\geq 2$, and if $\dim U_0\geq 5$, then
\begin{gather*}
s = \dim U_0 + \mathrm{wi}(W) \geq 5 + 2 = 7,
\end{gather*}
and we are done. So let $2<\dim U_0<5$.
\begin{enumerate}
\item[(i)]
First, let $\dim U_0=4$.
Assume $\rk B_1=\rk B_2=2$.
Because $C_3=-B_1^\top\tilde{I}B_2\neq 0$ is skew,
it is of rank $2$. So $\rk B_1=\rk B_2=2=\rk B_1^\top\tilde{I}B_2$.
By Lemma~\ref{lem_rk_dimU0_4}, the action of $\Gamma$ is not free.

Now assume $\rk B_1\geq 3$.
It follows from Lemma \ref{lem_rk_BX} that $\mathrm{wi}(W)\geq 3$ and
$\dim W \geq 6$, so once more
\begin{gather*}
s = \dim U_0 + \mathrm{wi}(W) \geq 4 + 3 = 7.
\end{gather*}
So the theorem holds for $\dim U_0=4$.
\item[(ii)]
Let $\dim U_0=3$.
If $\dim(\im B_1+\im B_2)\leq 5$, there exists a fixed point by Lemma \ref{lem_rk_dimU0_3}, so $\Gamma$ would not act freely.
So let $\dim(\im B_1+\im B_2)=6$:
As $[A_1,A_2]\neq 0$, the crossover rule
implies the existence of a pair of
columns $b_1^i$, $b_1^j$ in $B_1$ and a pair of columns $b_2^j$, $b_2^i$
in $B_2$ such that $\alpha=\langle b_1^i,b_2^j\rangle=-\langle b_1^j,b_2^i\rangle\neq 0$.
For simplicity say $i=1$, $j=2$.
The columns $b_1^1, b_1^2, b_1^3$ span the totally isotropic subspace $\im B_1$ of $W$, and $b_2^1, b_2^2, b_2^3$ span $\im B_2$.
We have a Witt decomposition with respect to
$W_{12}=\RR b_1^1\oplus \RR b_1^2$ (Lemma \ref{lem_rk_BX}),
\[
W = W_{12} \oplus W' \oplus W_{12}^*,
\]
where $W_{12}^*=\RR b_2^1\oplus \RR b_2^2$.
Because $b_1^3\perp\im B_1$ and $b_2^3\perp\im B_2$,
\[
b_1^3 = \lambda_1 b_1^1 + \lambda_2 b_1^2 +  b',\qquad
b_2^3 = \mu_1 b_2^1 + \mu_2 b_2^2 +  b'',
\]
where $b', b''\in W'$ are linearly independent because
$\dim(\im B_1+\im B_2)=6$.
From $0=\langle b_1^3,b_1^3\rangle$ it follows that $\langle b',b'\rangle=0$, and similarly $\langle b'',b''\rangle=0$.
The crossover rule then implies
\begin{gather*}
\lambda_1\langle b_2^2,b_1^1\rangle = \langle b_2^2,b_1^3\rangle = - \langle b_2^3,b_1^2\rangle = -\mu_1\langle b_2^1,b_1^2\rangle
=\mu_1\langle b_2^2,b_1^1\rangle,\\
\lambda_2\langle b_2^1,b_1^2\rangle = \langle b_2^1,b_1^3\rangle = - \langle b_2^3,b_1^1\rangle = -\mu_2\langle b_2^2,b_1^1\rangle
=\mu_2\langle b_2^1,b_1^2\rangle.
\end{gather*}
As the inner products are $\neq 0$, it follows that $\lambda_1=\mu_1$,
$\lambda_2=\mu_2$.
Then, by the duality rule,
\[
0 = \langle b_1^3,b_2^3\rangle = (\ub{\lambda_1\mu_2-\lambda_2\mu_1}{=0})\langle b_2^2,b_1^1\rangle
+\langle b',b''\rangle
=\langle b',b''\rangle.
\]
So $b'$ and $b''$ span a $2$-dimensional
totally isotropic subspace in the non-degenerate space $W'$,
so this subspace has a $2$-dimensional dual in $W'$ and $\dim W'\geq 4$,
$\mathrm{wi}(W')\geq 2$, follows.
Hence
\begin{gather*}
\mathrm{wi}(W) = \dim W_{12} + \mathrm{wi}(W') \geq 2+2 = 4,
\end{gather*}
and again
\begin{gather*}
s = \dim U_0 + \mathrm{wi}(W) \geq 3+4 = 7.
\end{gather*}
\end{enumerate}
In any case $s\geq 7$ and $n\geq 14$.
\end{proof}

\begin{cor}\label{cor_dimbound}
If $M$ is a complete flat homogeneous pseudo-Riemannian mani\-fold
such that its fundamental group $\Gamma$ has non-abelian linear holonomy group $\hol(\Gamma)$, then
\[
\dim M\geq 14
\]
and the signature $(n-s,s)$ of $M$ statisfies $n-s\geq s\geq 7$.
\end{cor}

The dimension bound in Corollary \ref{cor_dimbound} is sharp,
as the following example from \cite{BG} shows:

\begin{example}\label{example2}
Let $\Gamma\subset\Iso(\RR^{14}_7)$ be the group generated by
\[
\gamma_1 =
\Bigl(
\begin{pmatrix}
I_5 & -B_1^\top\tilde{I} & C_1 \\
0 & I_4 & B_1 \\
0 & 0 & I_5
\end{pmatrix},
\begin{pmatrix}
0\\ 0\\ u_1^*
\end{pmatrix}\Bigr),\quad
\gamma_2 =
\Bigl(
\begin{pmatrix}
I_5 & -B_2^\top\tilde{I} & C_2 \\
0 & I_4 & B_2 \\
0 & 0 & I_5
\end{pmatrix},
\begin{pmatrix}
0\\ 0\\ u_2^*
\end{pmatrix}\Bigr)
\]
in the basis representation (\ref{eq_nonabelian3}). Here,
\[
B_1 =
\begin{pmatrix}
-1 & 0 & 0 & 0 & 0\\
0 & -1 & 0 & 0 & 0\\
0 & -1 & 0 & 0 & 0\\
-1 & 0 & 0 & 0 & 0
\end{pmatrix},\quad
C_1 = \begin{pmatrix}
0 & 0 & 0 & 0 & 0 \\
0 & 0 & 0 & 0 & 0 \\
0 & 0 & 0 & 0 & -1 \\
0 & 0 & 0 & 0 & 0 \\
0 & 0 & 1 & 0 & 0
\end{pmatrix},\quad
u_1^* =
\begin{pmatrix}
0\\0\\0\\-1\\0
\end{pmatrix},
\]
\[
B_2 =
\begin{pmatrix}
0 & -1 & 0 & 0 & 0\\
1 & 0 & 0 & 0 & 0\\
-1 & 0 & 0 & 0 & 0\\
0 & 1 & 0 & 0 & 0
\end{pmatrix},\quad
C_2 = \begin{pmatrix}
0 & 0 & 0 & 0 & 0 \\
0 & 0 & 0 & 0 & 0 \\
0 & 0 & 0 & 0 & 0 \\
0 & 0 & 0 & 0 & -1 \\
0 & 0 & 0 & 1 & 0
\end{pmatrix},\quad
u_2^* =
\begin{pmatrix}
0\\0\\1\\0\\0
\end{pmatrix},
\]
and $\tilde{I}=\left(\begin{smallmatrix} I_2 & 0\\
0&-I_2\end{smallmatrix}\right)$.
Their commutator is
\[
\gamma_3 = [\gamma_1,\gamma_2] =
\Bigl(
\begin{pmatrix}
I_5 & 0 & C_3 \\
0 & I_4 & 0 \\
0 & 0 & I_5
\end{pmatrix},
\begin{pmatrix}
u_3\\ 0\\ 0
\end{pmatrix}\Bigr),
\]
with
\[
C_3 =
\begin{pmatrix}
0 & -4 & 0 & 0 & 0 \\
4 & 0 & 0 & 0 & 0 \\
0 & 0 & 0 & 0 & 0 \\
0 & 0 & 0 & 0 & 0 \\
0 & 0 & 0 & 0 & 0
\end{pmatrix},\quad
u_3 =
\begin{pmatrix}
0\\0\\0\\0\\2
\end{pmatrix}.
\]
The group $\Gamma$ is isomorphic to a discrete Heisenberg group,
and the linear parts of $\gamma_1,\gamma_2$ do not commute.
In \cite[Example 6.4]{BG} it was shown that $\Gamma$ has transitive
centralizer in $\Iso(\RR^{14}_7)$ and acts properly discontinuously and
freely on $\RR^{14}_7$. Hence $M=\RR^{14}_7/\Gamma$ is a complete flat
pseudo-Riemannian homogeneous manifold of dimension $14$ with
non-abelian linear holonomy.
\end{example}

%%%%%%%%%%%%%%%%%%%%%%%%%%%%%%%%%%%%%%%%%%%%%%%%%%%%%%%%%%%%%%%%%%%%%%
\section{Low Dimensions}
\label{chp_lowdim}
%%%%%%%%%%%%%%%%%%%%%%%%%%%%%%%%%%%%%%%%%%%%%%%%%%%%%%%%%%%%%%%%%%%%%%

In this section, we determine the structure of the fundamental groups
of complete flat pseudo-Riemannian homogeneous spaces $M$ of
dimensions $\leq 6$ and of those with signature $(n-2,2)$.
The signatures $(n,0)$, $(n-1,1)$ and $(n-2,2)$ were already studied by Wolf \cite[Corollary 3.7.13]{Wolf_buch}.
In particular, he derived the following:

\begin{prop}[Wolf]\label{prop_lorentz_complete}
If $M$ is a complete homogeneous flat Riemannian or Lorentzian manifold,
then the fundamental group of $M$ is an abelian group consis\-ting of pure
translations.
\end{prop}

Let $\Gamma\subset\Iso(\RR^n_s)$ denote the fundamental group of $M$ and
$G\subset\Iso(\RR^n_s)$ its real Zariski closure with Lie
algebra $\frg$.
Let $U_\Gamma$ be as in Proposition \ref{prop_wolf1}.

We start by collecting some general facts about $\Gamma$ and $G$.

\begin{remark}
$G\subset\Iso(\RR^n_s)$ is unipotent, hence simply connected.
Then $\Gamma$
is finitely generated
and torsion-free by \cite[Theorem 2.10]{ragh},
as it is a discrete subgroup of $G$.
Further, $\rk \Gamma=\dim G$.
\end{remark}

The fundamental theorem for finitely generated abelian groups
states:

\begin{lem}\label{lem_free_abelian}
If $\Gamma$ is abelian and torsion-free, then $\Gamma$ is free abelian.
\end{lem}

By \cite[Theorem 5.1.6]{corwin}, there exists a Malcev basis of $G$
which generates $\Gamma$. We shall call it a
\emph{Malcev basis of $\Gamma$}.

\begin{lem}\label{lem_lin_indep}
Let $\gamma_1,\ldots,\gamma_k$ denote a Malcev basis of $\Gamma$.
If $M$ is complete, then the translation parts $v_1,\ldots,v_k$ of the
$\gamma_i=(I+A_i,v_i)$ are linearly independent.
\end{lem}
\begin{proof}
$\Gamma$ has transitive centralizer in $\Iso(\RR^n_s)$. By continuity, so
does $G$. Hence $G$ acts freely on $\RR^n_s$.
Then the orbit map
$G\to\RR^n_s$, $g\mapsto g.0$, at the point $0$
is a diffeomorphism onto the orbit $G.0$. Because $G$ acts by affine
transformations, $G.0$ is the span of the tranlsation parts of
the $\gamma_i$.
So
\[
k = \rk\Gamma = \dim G = \dim G.0 = \dim\mathrm{span}\{v_1,\ldots,v_k\}.
\]
So the $v_i$ are linearly independent.
\end{proof}

%%%
\subsection{Signature $(n-2,2)$}

As always, we assume $n-2\geq  2$.

\begin{prop}[Wolf]\label{prop_sig2_complete}
Let $M=\RR^n_2/\Gamma$ be a flat pseudo-Rie\-mannian homogeneous 
manifold.
Then $\Gamma$ is a free abelian group.
In particular, the fundamental group of every flat
pseudo-Riemannian homogeneous manifold $M$ of dimension $\dim M\leq 5$
is free abelian.
\end{prop}
\begin{proof}
It follows from Corollary \ref{cor_dimbound} that $\Gamma$ has abelian holonomy.
Consequently, if $\gamma=(I+A,v)\in\Gamma$ such that $A\neq 0$, then
\[
A=\begin{pmatrix}
0 & 0& C\\
0& 0 & 0\\
0&0&0
\end{pmatrix}
\]
in a Witt basis with respect to $U_{\Gamma}$.
Here, $C\neq 0$ is a skew-symmetric $2\times 2$-matrix, so we have $\rk A= 2$.
Because $\im A\subset U_{\Gamma}$ and both these spaces are
totally isotropic, we have $\dim\im A=\dim U_{\Gamma}=2$.
Because $Av=0$ we get $v\in\ker A=(\im A)^\bot=U_{\Gamma}^\bot$.
But then $Bv=0$ for any $(I+B,w)\in\Gamma$, and as also $BA=0$,
it follows that
$[(I+B,w),(I+A,v)]=(I+2BA,2Bv)=(I,0)$. Hence $\Gamma$ is abelian.
It is free abelian by Lemma \ref{lem_free_abelian}.
\end{proof}

In the remainder of this section, the group $\Gamma$ is always abelian, so
the space $U_{\Gamma}=\sum_A \im A$ is totally isotropic
(in particular, $U_0=U_\Gamma$).
We fix a Witt decomposition with respect to $U_{\Gamma}$,
\[
\RR^{n}_2=U_{\Gamma}\oplus W\oplus U_{\Gamma}^*
\]
and any $v\in\RR^{n}_2$ decomposes into $v=u+w+u^*$ with $u\in U_{\Gamma}$,
$w\in W$, $u^*\in U_{\Gamma}^*$.

\begin{remark}\label{rem_dimUg_2}
As seen in the proof of Proposition \ref{prop_sig2_complete},
if $\dim U_{\Gamma}=2$, then $U_{\Gamma}=\im A$
for any $\gamma=(I+A,v)$ with $A\neq0$.
\end{remark}

We can give a more precise description of the elements of $\Gamma$:

\begin{prop}\label{prop_sig2_fg}
Let $M=\RR^{n}_2/\Gamma$ be a complete flat pseudo-Rie\-mannian homo\-geneous 
manifold. Then:
\begin{enumerate}
\item
$\Gamma$ is generated by elements $\gamma_i=(I+A_i,v_i)$, $i=1,\ldots,k$,
with linearly independent translation parts $v_1,\ldots,v_k$.
\item
If there exists $(I+A,v)\in\Gamma$ with $A\neq 0$, then in a Witt basis with
respect to $U_{\Gamma}$,
\begin{equation}
\gamma_i = (I+A_i,v_i)
=\biggl(\begin{pmatrix}
I_2 & 0 & C_i \\
0 & I_{n-4} & 0 \\
0 & 0 & I_2
\end{pmatrix},
\begin{pmatrix}
u_i\\
w_i\\
0
\end{pmatrix}\biggr)
\label{eq_sig2}
\end{equation}
with
$C_i=\left(\begin{smallmatrix}
0 & c_i\\
-c_i & 0
\end{smallmatrix}\right)$, $c_i\in\RR$,
$u_i\in\RR^2$,
$w_i\in\RR^{n-4}$.
\item
$\sum_i \lambda_i w_i = 0$ implies $\sum_i \lambda_i C_i=0$
(equivalently $\sum_i\lambda_i A_i=0$)
for all $\lambda_1,\ldots,\lambda_k\in\RR$.
\end{enumerate}
\end{prop}
\begin{proof}
We know from Proposition \ref{prop_sig2_complete} that $\Gamma$ is free
abelian. Let $\gamma_1,\ldots,\gamma_k$ denote a minimal set of generators 
with $\gamma_i=(I+A_i,v_i)$.
\begin{enumerate}
\item
Lemma \ref{lem_lin_indep}.
\item
If $A\neq 0$ exists, then $U_{\Gamma}=\im A$ is a 2-dimensional totally
isotropic subspace.
The matrix representation (\ref{eq_sig2})
is known from the proof of Proposition \ref{prop_sig2_complete}.
As $\Gamma$ is abelian, we have
$A_i v_j=0$ for all $i,j$. So $v_j\in\bigcap_i\ker A_i=U_{\Gamma}^\bot$ for all $j$.
\item
Assume $\sum_i \lambda_i w_i=0$ and set $C=\sum_i \lambda_i C_i$.
Then $\sum_i \lambda_i (A_i,v_i)=(A,u)$, where $u\in U_{\Gamma}$.
If $A\neq 0$, then $ G$ would have
a fixed point (see Corollary \ref{cor_dimU0_2}).
So $A=0$, which implies $C=0$.
\end{enumerate}
\end{proof}

Conversely, every group of the form described in the previous proposition
defines a homogeneous space:

\begin{prop}\label{prop_sig2_fg2}
Let $U$ be a 2-dimensional totally isotropic subspace of $\RR^{n}_2$,
and let $\Gamma\subset\Iso(\RR^{n}_2)$ be a subgroup generated by affine
transformations $\gamma_1,\ldots,\gamma_k$ of the form (\ref{eq_sig2})
with linearly independent translation parts.
Further, assume that $\sum_i \lambda_i w_i = 0$ implies $\sum_i \lambda_i C_i=0$
(equivalently $\sum_i\lambda_i A_i=0$) for all $\lambda_1,\ldots,\lambda_k\in\RR$.
Then $\RR^{n}_2/\Gamma$ is a complete flat pseudo-Riemannian homogeneous
manifold.
\end{prop}
\begin{proof}\
\begin{enumerate}
\item[(i)]
From the matrix form (\ref{eq_sig2}) it follows that $\Gamma$ is free abelian,
and the linear independence of the translation parts implies that it is
a discrete subgroup of $\Iso(\RR^{n}_2)$.
\item[(ii)]
We check that the centralizer of $\Gamma$ in $\Iso(\RR^{n}_2)$
acts transitively:
Let $\iso(\RR^n_s)$ denote the Lie algebra of $\Iso(\RR^n_s)$.
In the given Witt basis, the following are elements of $\iso(\RR^n_s)$:
\[
S=\biggl(\begin{pmatrix}
0 & -B^\top & 0 \\
0 & 0 & B \\
0 & 0 & 0
\end{pmatrix},
\begin{pmatrix}
x\\
y\\
z
\end{pmatrix}\biggr),
\quad
x,z\in\RR^2,y\in\RR^{n-2}.
\]
Now assume arbitrary $x,y,z$ are given.
We will show that we can determine $B$ such that
$S$ centralises $\log(\Gamma)$.
Writing out the commutator equation $[S,(A_i,v_i)]$ blockwise,
we see that $[S,(A_i,v_i)]=0$ is equivalent to
\[
-B^\top w_i = C_i z.
\]
For simplicity, assume that $w_1,\ldots,w_j$ form a maximal linearly independent
subset of $w_1,\ldots,w_k$ ($j\leq k$).
As $-B^\top$ is a $2\times(n-2)$-matrix, the linear system
\begin{align*}
-B^\top w_1 &= C_1 z \\
&\vdots \\
-B^\top w_j &= C_j z
\end{align*}
consists of $2j$ linearly independent equations and $2(n-2)$ variables
(the entries of $B$).
As $\dim W=n-2\geq j$, this system is always solvable.

So $S$ can be determined such that it commutes with
$\gamma_1,\ldots,\gamma_j$. It remains to check that $S$ also commutes with
$\gamma_{j+1},\ldots,\gamma_k$. By assumption, each $w_l$ ($l>j$) is a linear
combination $w_l=\sum_{i=1}^j\lambda_i w_i$.
Now $w_l-\sum_{i=1}^j \lambda_i w_i=0$ implies
$C_l-\sum_{i=1}^j \lambda_i C_i=0$. But this means
\[
-B^\top w_l
=\sum_{i=1}^j \lambda_i (\ub{-B^\top w_i}{=C_i z})
=\bigl(\sum_{i=1}^j \lambda_i C_i\bigr) z
= C_l z,
\]
so $[(A_l,v_l),S]=0$.

The elements $\exp(S)$ generate a unipotent subgroup of the
centralizer of $\Gamma$, so its open orbit at $0$ is closed
by \cite[Proposition 4.10]{borel}.
As $x,y,z$ can be chosen arbitrarily, its tangent space at $0$ is
$\RR^{n}_2$.
Hence the orbit of the centralizer at $0$ is open and closed,
and therefore it is all of $\RR^{n}_2$.
Consequently, $\Gamma$ has transitive centralizer.
\item[(iii)]
%The action of $ G$ at $0$ is free.
%In fact, if $(\sum_i \lambda_i A_i,0)\in\frg$, it follows from our
%assumptions that $\sum_i\lambda_i C_i=0$, hence $\sum_i\lambda_i A_i=0$.
Because the centralizer is transitive, the action free everywhere.
It follows from \cite[Proposition 7.2]{BG} that $\Gamma$
acts properly discontinuously.
\end{enumerate}
Now $\RR^{n}_2/\Gamma$ is a complete homogeneous manifold due to the
transitive action of the centralizer on $\RR^n_2$.
\end{proof}

%%%
\subsection{Dimension $\leq 5$}
\label{sec_dim_leq5}

\begin{prop}[Wolf]\label{prop_dimM_4}
Let $M=\RR^{n}_s/\Gamma$ be a complete homogeneous flat
pseudo-Riemannian manifold of dimension $\leq 4$.
Then $\Gamma$ is a free abelian group consisting of pure translations.
\end{prop}
For a proof, see \cite[Corollary 3.7.11]{Wolf_buch}.

\begin{prop}\label{prop_dimM_5}
Let $M=\RR^{5}_s/\Gamma$ be a complete homogeneous flat
pseudo-Rie\-mannian manifold of dimension $5$.
Then $\Gamma$ is a free abelian group.
Depending on the signature of $M$, we have the following possibilities:
\begin{enumerate}
\item
Signature $(5,0)$ or $(4,1)$: $\Gamma$ is a group of pure translations.
\item
Signature $(3,2)$: $\Gamma$ is either a group of pure translations, or
there exists $\gamma_1=(I+A_1,v_1)\in\Gamma$ with $A_1\neq0$.
In the latter case, $\rk\Gamma\leq 3$, and if $\gamma_1,\ldots,\gamma_k$ ($k=1,2,3$) are
generators
of $\Gamma$, then $v_1,\ldots,v_k$ are linearly independent, and
$w_i=\frac{c_i}{c_1}w_1$ in the notation of (\ref{eq_sig2}) ($i=1,\ldots,k$).
\end{enumerate}
\end{prop}
\begin{proof}
$\Gamma$ is free abelian by Proposition \ref{prop_sig2_complete}.
The statement for signatures $(5,0)$ and $(4,1)$
follows from Propo\-sition \ref{prop_lorentz_complete}.

Let the signature be $(3,2)$ and assume $\Gamma$ is not a group of pure 
trans\-lations. Then $U_{\Gamma}=\im A$ is 2-dimensional
(where $(I+A,v)\in\Gamma$, $A\neq 0$).
By Lemma \ref{lem_lin_indep}, the translation parts of
the generators of $\Gamma$ are linearly independent elements of
$U_{\Gamma}^\bot$, which is 3-dimensional. So $\rk\Gamma\leq 3$.
Now, $U_{\Gamma}^\bot=U_{\Gamma}\oplus W$ with $\dim W=1$.
So the $W$-components of the translation parts are multiples of each other,
and it follows from
part (c) of Proposition \ref{prop_sig2_fg} that $w_1\neq 0$ and
$w_i=\frac{c_i}{c_1}w_1$.
\end{proof}

%%%
\subsection{Dimension $6$}
\label{sec_dim_equal6}

In dimension $6$, both abelian and non-abelian $\Gamma$ exist.

We introduce the following notation:
For $x\in\RR^3$, let
\[
T(x) = \begin{pmatrix}
0 & -x_3 & x_2\\
x_3 & 0 & -x_1\\
-x_2 & x_1 & 0
\end{pmatrix}.
\]
Then for any $y\in\RR^3$,
\[
T(x)y = x\times y,
\]
where $\times$ denotes the vector cross product on $\RR^3$.
\index{$T(x)$ (cross product matrix)}

\begin{lem}\label{lem_crossproduct}
Let $\Gamma\in\Iso(\RR^6_3)$ be a group with transitive
centralizer in $\Iso(\RR^6_3)$.
An element $X\in\log(\Gamma)$ has the form
\begin{equation}
X = \bigl(\begin{pmatrix}
0 & C\\
0 & 0
\end{pmatrix},
\begin{pmatrix} u\\ u^*\end{pmatrix}
\bigr)
\label{eq_nonabelian_33}
\end{equation}
with respect to the Witt decomposition $\RR^{6}_3=U_{\Gamma}\oplus U_{\Gamma}^*$.
Furthermore,
\[
C = \alpha_X T(u^*)
\]
for some $\alpha_X\in\RR$. If $[X_1,X_2]\neq 0$ for $X_1,X_2\in\log(\Gamma)$,
then $\alpha_{X_1}=\alpha_{X_2}\neq 0$.
\end{lem}
\begin{proof}
The holonomy is abelian by Corollary \ref{cor_dimbound}, so
(\ref{eq_nonabelian_33}) follows.

For $X\in\log(\Gamma)$ we have $C u^*=0$, that is, for any $\alpha\in\RR$,
\[
C u^* = \alpha u^*\times u^* = 0.
\]
If $X$ is non-central, then $C\neq0$ and $u^*\neq 0$.
Now let $x,y\in\RR^3$ such that $u^*,x,y$ form a basis of $\RR^3$.
Because $C$ is skew,
\[
u^{*\top} C x = - u^{*\top} C^\top x = -(C u^*)^\top x = 0.
\]
Also,
\[
x^\top C x = - x^\top C^\top x\quad\text{ and }\quad
x^\top C x = (x^\top C x)^\top = x^\top C^\top x,
\]
hence $x^\top C x=0$. So $C x$ is perpendicular to the span of $x,u^*$
in the Euclidean sense\footnote{That is, with respect to the canonical
positive definite inner product on $\RR^3$.}.
This means there is a $\alpha\in\RR$ such that
\[
C x = \alpha u^*\times x.
\]
In the same way we get $C y = \beta u^*\times y$ for some $\beta\in\RR$.
As neither $x$ nor $y$ is in the kernel of $C$ (which is spanned by $u^*$),
$\alpha,\beta\neq 0$.

As $y$ is not in the span of $u^*,x$, we have
\begin{align*}
0 \neq\ & x^\top C y = \beta x^\top (u^*\times y) \\
&=- y^\top C x = -\alpha y^\top (u^*\times x) = -\alpha x^\top (y\times u^*)
=\alpha x^\top (u^*\times y),
\end{align*}
where the last line uses standard identities for the vector product.
So $\alpha=\beta$, and $C$ and $\alpha T(u^*)$ coincide on a basis of $\RR^3$.

Now assume $[X_1,X_2]\neq 0$.
Then
\[
\alpha_2 u_2^* \times u_1^* = C_2 u_1^* = -C_1 u_2^*
= -\alpha_1 u_1^*\times u_2^* = \alpha_1 u_2^* \times u_1^*,
\]
and this expression is $\neq 0$ because $C_1 u_2^*$ is the
translation part of $(\frac{1}{2}[X_1,X_2])\neq 0$.
So $\alpha_1=\alpha_2$.
\end{proof}

\begin{prop}\label{prop_dimM_6}
Let $M=\RR^{6}_s/\Gamma$ be a complete homogeneous flat
pseudo-Rie\-mannian manifold of dimension $6$, and assume $\Gamma$ is abelian.
Then $\Gamma$ is free abelian.
Depending on the signature of $M$, we have the following possibilities:
\begin{enumerate}
\item
Signature $(6,0)$ or $(5,1)$:
$\Gamma$ is a group of pure translations.
\item
Signature $(4,2)$:
$\Gamma$ is either a group of pure translations, or
$\Gamma$ contains elements $\gamma=(I+A,v)$ with $A\neq 0$ subject to the
constraints of Proposition \ref{prop_sig2_fg}.
Further, $\rk\Gamma\leq 4$.
\item
Signature $(3,3)$:
If $\dim U_{\Gamma}<3$, then $\Gamma$ is one of the groups that may appear for signature $(4,2)$.
There is no abelian $\Gamma$ with $\dim U_{\Gamma}=3$.
\end{enumerate}
\end{prop}
\begin{proof}
$\Gamma$ is free abelian by Lemma \ref{lem_free_abelian}. The statement for 
signatures $(6,0)$ and $(5,1)$ follows from Proposition \ref{prop_lorentz_complete}.

If the signature ist $(4,2)$ and $\Gamma$ is not a group of pure translations,
then the statement follows from Proposition \ref{prop_sig2_fg}.
In this case, $U_{\Gamma}^\bot$ contains the linearly independent
translation parts and is of dimension $4$. So $\rk\Gamma\leq 4$.

Consider signature $(3,3)$.
If $\dim U_{\Gamma}=0$ or $=2$, then $\Gamma$ is a group as in the case for 
signature $(4,2)$. Otherwise, $\dim U_{\Gamma}=3$. We show that in the latter
case the centralizer of $\Gamma$ does not act with open orbit:
Any $\gamma\in\Gamma$ can be written as
\[
\gamma=(I+A,v)=
\bigl(\begin{pmatrix}
I_3 & C \\
0 & I_3
\end{pmatrix},
\begin{pmatrix}
u\\ u^*
\end{pmatrix}\bigr),
\]
where $C\in\sso_3$ and $u,u^*\in\RR^3$. In fact, we have
$\RR^{6}_3=U_{\Gamma}\oplus U_{\Gamma}^*$ and $U_{\Gamma}^\bot=U_{\Gamma}$.

We will show that $u^*=0$:
\begin{itemize}
\item[(i)]
Because $\rk C=2$ for every $C\in\sso_3$, $C\neq0$, but $U_{\Gamma}=\sum\im A$ is
$3$-dimensional, there exist $\gamma_1, \gamma_2\in\Gamma$ such that the skew
matrices $C_1$ and $C_2$ are linearly independent.
So, for every $u^*\in U_{\Gamma}^*$, there is an element $\gamma=(I+A,v)$
such that $A u^*\neq 0$.
\item[(ii)]
$\Gamma$ abelian implies $A_1 u_2^*=0$ for every
$\gamma_1,\gamma_2\in\Gamma$. With (i), this implies $u_2^*=0$.
So the translation part of every $\gamma=(I+A,v)\in\Gamma$ is an element
$v=u\in U_{\Gamma}$.
\end{itemize}
Step (ii) implies $C_1=\alpha_1 T(u_1^*)=0$ by Lemma \ref{lem_crossproduct},
but $C_1\neq0$ was required in step (i). Contradiction; so $\Gamma$
cannot be abelian.
\end{proof}

\begin{prop}\label{prop_dimM_6_nonab}
Let $M=\RR^{6}_s/\Gamma$ be a complete homogeneous flat
pseudo-Rie\-mannian manifold of dimension $6$, and assume $\Gamma$ is non-abelian.
Then the signature of $M$ is $(3,3)$, and $\Gamma$ is one of the following:
\begin{enumerate}
\item
$\Gamma=\Lambda\times\Theta$, where $\Lambda$ is a discrete Heisenberg group\index{Heisenberg group}
and $\Theta$ a discrete group of pure translations in $U_\Gamma$.
Then $3\leq\rk\Gamma=3+\rk\Theta\leq5$.
\item
$\Gamma$ is discrete group of rank $6$ with center
$\mathrm{Z}(\Gamma)=[\Gamma,\Gamma]$ of rank $3$.
In this case, $M$ is compact.
\end{enumerate}
\end{prop}
\begin{proof}
If the signature was anything but $(3,3)$ or $\dim U_0<3$,
then $\Gamma$ would have to be
abelian due Proposition \ref{prop_sig2_complete}.
The holonomy is abelian by Corollary \ref{cor_dimbound}.

For the following it is more convenient to work with the real
Zariski closure $ G$ of $\Gamma$ and its Lie algebra $\frg$.
As $\frg$ is 2-step nilpotent,
$\frg=\frv\oplus\frz(\frg)$, where $\frv$ is a vector subspace of $\frg$
of dimension $\geq 2$ spanned by non-central elements.
Set $\frv_{\Gamma}=\frv\cap\log(\Gamma)$.
We proceed in four steps:
\begin{enumerate}
\item[(i)]
Assume there are $X_i=(A_i,v_i)\in\frv$, $\lambda_i\in\RR$,
$v_i=u_i+u_i^*$ (for $i=1,\ldots,m$), such that $\sum_i \lambda_i u_i^*=0$.
Then
$\sum_i\lambda_i X_i=(\sum_i \lambda_i A_i,\sum_i\lambda_i v_i)=(A,u)\in\frv$,
where $u\in U_\Gamma$.
For all $(A',v')\in\frg$, the commutator with $(A,u)$ is
$[(A',v'),(A,u)] = (0,2A'u) = (0,0)$.
Thus $(A,u)\in\frv\cap\frz(\frg)=\{0\}$.

So if $X_1,\ldots,X_m\in\frv$ are linearly independent,
then $u_1^*,\ldots,u_m^*\in U_0^*$ are linearly independent (and by
Lemma \ref{lem_crossproduct} the $C_1,\ldots,C_m$ are too).
But $\dim U_0^*=3$, so $\dim\frv\leq 3$.
\item[(ii)]
If $Z\in\frz(\frg)$, then $C_Z=0$ and $u_Z^*=0$:
As $Z$ commutes with $X_1, X_2$, we have
$C_Z u_1^* = 0 = C_Z u_2^*$.
By step (i), $u_1^*, u_2^*$ are linearly  independent.
So $\rk C_Z<2$, which implies $C_Z=0$ because $C_Z$ is a skew
$3\times 3$-matrix.
Also, $C_1 u_Z^*=0=C_2 u_Z^*$, so $u_Z^*=\ker C_1\cap\ker C_2=\{0\}$.
So $\exp(Z)=(I,u_Z)$ is a translation by $u_Z\in\Gamma$.
\item[(iii)]
Assume $\dim\frv=2$. Let $\frv$ be spanned by $X_1,X_2$, and $Z_{12}=[X_1,X_2]$
is a pure translation by an element of $U_\Gamma$.
The elements $X_1,X_2,Z_{12}$ span a Heisenberg algebra $\frh_3$
contained in $\frg$.
If $\dim\frg>3$, then $\frz(\frg)=\RR Z_{12}\oplus\frt$, where
according to step (ii)
$\frt$ is a subalgebra of pure translations by elements of
$U_\Gamma$. So $\frg=\frh_3\oplus\frt$ with
$0\leq\dim\frt<\dim U_\Gamma=3$.
This gives part (a) of the proposition.
\item[(iv)]
Now assume $\dim\frv=3$. We show that $\frz(\frg)=[\frv,\frv]$ and
$\dim\frz(\frg)=3$:
Let $X_1=(A_1,v_1), X_2=(A_2,v_2)\in\frv_{\Gamma}$
such that $[X_1,X_2]\neq 0$.
By Lemma \ref{lem_crossproduct}, $C_1=\alpha T(u_1^*)$ and
$C_2=\alpha T(u_2^*)$ for some number $\alpha\neq 0$.
There exists $X_3\in\frv_{\Gamma}$ such that
$X_1,X_2,X_3$ form basis of $\frv$. By step (i), $u_1^*,u_2^*,u_3^*$ are
li\-nearly independent. For $i=1,2$, $\ker C_i=\RR u_i^*$, and
$u_3^*$ is proportional to neither $u_1^*$ nor $u_2^*$.
This means $C_1 u_3^*\neq0\neq C_2 u_3^*$,
which implies $[X_1,X_3]\neq0\neq[X_2,X_3]$.
By Lemma \ref{lem_crossproduct}, $C_3=\alpha T(u_3^*)$.

Write $Z_{ij}=[X_i,X_j]$.
The non-zero entries of the translation parts of the commutators $Z_{12}$,
$Z_{13}$ and $Z_{23}$ are
\[
C_1 u_2^*=\alpha u_1^*\times u_2^*,\qquad
C_1 u_3^*=\alpha u_1^*\times u_3^*,\qquad
C_2 u_3^*=\alpha u_2^*\times u_3^*.
\]
Linear independence of $u_1^*,u_2^*,u_3^*$ implies that these
are linearly independent.
Hence the commutators $Z_{12},Z_{13},Z_{23}$ are linearly independent
in $\frz(\frg)$. Because $\dim\frg=\dim\frv+\dim\frz(\frg)\leq 6$, it follows that 
$\frz(\frg)$ is spanned by these $Z_{ij}$,
that is $\frz(\frg)=[\frv,\frv]$.
This gives part (b) of the proposition.
\end{enumerate}
\end{proof}

\begin{remark}
In case (2) of Proposition \ref{prop_dimM_6_nonab} it can be shown
that $\Gamma$ is a lattice in a Lie group
$\mathrm{H}_3\ltimes_{\Ad*}\frh_3^*$, see \cite[Section 5.3]{globke}.
\end{remark}

We have a converse statement to Proposition \ref{prop_dimM_6_nonab}:

\begin{prop}\label{prop_dimM_6_nonab2}
Let $\Gamma$ be a subgroup of $\Iso(\RR^{6}_3)$.
Then $M=\RR^{6}_3/\Gamma$ is a complete flat pseudo-Riemannian homogeneous 
manifold if there exists a $3$-dimensional totally isotropic subspace $U$ and
$\Gamma$ is a group of type (1) or
(2) in Proposition \ref{prop_dimM_6_nonab} (with $U_\Gamma$ replaced by
$U$).
\end{prop}
\begin{proof}
Both cases can be treated simultaneously.

Let $X_1,X_2,X_3\in\log(\Gamma)$ such that the $\exp(X_i)$
generate $\Gamma$.
The number $\alpha\neq0$
from Lemma \ref{lem_crossproduct} is necessarily the same for
$X_1,X_2,X_3$.
\begin{enumerate}
\item[(i)]
The group $\Gamma$ is discrete because the translation parts of the generators
$\exp(X_i)$ and those of the generators of $\mathrm{Z}(\Gamma)$ form a
linearly independent set.
\item[(ii)]
We show that the centralizer of $\Gamma$ is transitive.
Consider the following elements
\[
S=\bigl(\begin{pmatrix}
0 & -\alpha T(z) \\
0 & 0
\end{pmatrix},
\begin{pmatrix}
x\\ z
\end{pmatrix}\bigr)\in\iso(\RR^{6}_3)
\]
with $x,z\in\RR^3$ arbitrary. 
Then $[X_i,S]=0$ for $i=1,2,3$, because
\[
C_i z = \alpha u_i^*\times z = -\alpha z\times u_i^* = -\alpha T(z) u_i^*.
\]
Clearly, $S$ also commutes with any translation by a vector from $U$.
So in both cases (1) and (2), $\Gamma$ has a centralizer with an open orbit
at $0$.
The exponentials of the elements of $S$ clearly generate a unipotent subgroup
of $\Iso(\RR^{6}_3)$, hence the open orbit is also closed and thus all of
$\RR^{6}_3$.
\item[(iii)]
From the transitivity of the centralizer, it also follows that the action
is free and thus properly discontinuous (\cite[Proposition 7.2]{BG}).
\end{enumerate}
So $\RR^6_3/\Gamma$ is a complete homogeneous manifold.
\end{proof}

In the situation of Proposition \ref{prop_dimM_6_nonab} it is natural to ask
whether the statement can be simplified by claiming that $\Gamma$ is always
a subgroup of a group of type (2) in Proposition \ref{prop_dimM_6_nonab}.
But this is not always the case:

\begin{example}\label{exp_no_butterfly}
We choose the generators $\gamma_i=(I+A_i,v_i)$, $i=1,2$,
of a discrete Heisenberg group $\Lambda$ as follows:
If we decompose $v_i=u_i+u_i^*$ where $u_i\in U_\Gamma$, $u_i^*\in U_\Gamma^*$, let
$u_i=0$, $u_i^*=e_i^*$, $\alpha=1$ (with $\alpha$ as in the proof
of Proposition \ref{prop_dimM_6_nonab} and $e_i^*$ refers to the $i$th unit vector
taken as an element of $U_\Gamma^*$).
Then $\gamma_3=[\gamma_1,\gamma_2]=(I,v_3)$, where $u_3=e_3$, $u_3^*=0$.
Let $\gamma_4=(I,u_4)$ be the trans\-lation by
$u_4=\sqrt{2}e_1+\sqrt{3}e_2\in U_\Gamma$.
Let $\Theta=\langle\gamma_4\rangle$ and $\Gamma=\Lambda\cdot\Theta$ ($\cong\Lambda\times\Theta$).

Assume there exists $X=(A,v)$ of the form (\ref{eq_nonabelian_33})
not commuting with $X_1,X_2$.
Then the respective translation parts of $[X_1,X]$ and $[X_2,X]$ are
\[
e_1\times u^*
=\begin{pmatrix}
0\\-\eta_3\\\eta_2
\end{pmatrix},
e_2\times u^*
=\begin{pmatrix}
\eta_3\\0\\-\eta_1
\end{pmatrix}\in U_\Gamma
\]
where $\eta_i$ are the components of $u^*$, and $\eta_3\neq 0$ due to the fact
that $X$ and the $X_i$ do not commute. If $\Gamma$ could be embedded into
into a group of type (2),
such $X$ would have to exist.
But by construction $u_4$ is not contained in the $\ZZ$-span of
$e_3, e_1\times u^*, e_2\times u^*$.
So the group generated by $\Gamma$ and $\exp(X)$ is not discrete
in $\Iso(\RR^{6}_3)$.
\end{example}

%%%%%%%%%%%%%%%%%%%%%%%%%%%%%%%%%%%%%%%%%%%%%%%%%%%%%%%%%%%%%%%%%%%%%%
\section{Fundamental Groups of Complete Flat Pseudo-Riemannian Homogeneous Spaces}
\label{chp_fg}
%%%%%%%%%%%%%%%%%%%%%%%%%%%%%%%%%%%%%%%%%%%%%%%%%%%%%%%%%%%%%%%%%%%%%%

In this section we will prove the following:

\begin{thm}\label{thm_fg}
Let $\Gamma$ be a finitely generated torsion-free 2-step nilpotent
group of rank $n$.
Then there exists a faithful representation
$\varrho:\Gamma\to\Iso(\RR^{2n}_n)$ such that
$M=\RR^{2n}_n/\varrho(\Gamma)$ a complete flat pseudo-Riemannian
homogeneous mani\-fold $M$ of signature $(n,n)$ with abelian linear
holonomy group.
\end{thm}

We start with a construction given in \cite[Paragraph 5.3.2]{Baues}
to obtain nilpotent Lie groups with flat bi-invariant metrics.
Let $\frg$ be a real 2-step nilpotent Lie algebra of finite dimension $n$.
Then the semidirect sum $\frh=\frg\oplus_{\ad^*}\frg^*$ is a
2-step nilpotent Lie algebra with Lie product
\begin{equation}
[(X,\xi),(Y,\eta)] = ([X,Y], \ad^*(X)\eta-\ad^*(Y)\xi),
\label{eq_lie_product}
\end{equation}
where $X,Y\in\frg$, $\xi,\eta\in\frg^*$ and $\ad^*$ denotes the
coadjoint representation.
An invariant inner product on $\frh$ is given by
\begin{equation}
\langle (X,\xi),(Y,\eta)\rangle = \xi(Y)+\eta(X).
\label{eq_inner_product}
\end{equation}
Its signature is $(n,n)$, as the subspaces $\frg$ and $\frg^*$ are
totally isotropic and dual to each other.

If $G$ is a simply connected 2-step nilpotent Lie group with Lie algebra
$\frg$, then $H=G\ltimes_{\Ad^*}\frg^*$ (with $\frg^*$ taken as a vector
group) is a simply connected 2-step nilpotent Lie group with
Lie algebra $\frh$, and $\langle\cdot,\cdot\rangle$ induces a
bi-invariant flat pseudo-Riemannian metric on $H$.

\begin{remark}
For any lattice $\Gamma_H\subset H$, the space $H/\Gamma_H$ is a compact
flat pseudo-Riemannian homogeneous manifold. In particular, $H$ is 
complete (see \cite[Proposition 9.39]{oneill}).
By \cite[Theorem 3.1]{BG}, $\Gamma_H$ has abelian linear holonomy.
\end{remark}

\begin{proof}[Proof of Theorem \ref{thm_fg}.]
Let $\Gamma$ be a finitely generated torsion-free 2-step nilpotent group.
The real Malcev hull $G$ of $\Gamma$ is a 2-step nilpotent simply
connected Lie group such that $\Gamma$ is a lattice in $G$.
In particular, $\rk\Gamma=\dim G=n$.
If $\frg$ is the Lie algebra of $G$, let $H$ be as in the construction
above. We identify $G$ with the closed subgroup $G\times\{0\}$ of $H$.
As $\Gamma$ is a discrete subgroup of $H$, it follows from the remark
above that $M=H/\Gamma$
is a complete flat pseudo-Riemannian homogeneous manifold with abelian
linear holonomy.

As $H$ has signature $(n,n)$, the development representation $\varrho$
of the right-multi\-plication of $G$ gives the representation of
$\Gamma$ as isometries of $\RR^{2n}_n$.
\end{proof}

%We show that $\hol(\Gamma)$ is abelian via the representation 
%(\ref{eq_nonabelian3}) for $\frg$.
%As $\frg$ is 2-step nilpotent, there is a subspace
%$\frw\subset\frg$ such that $\frg=\frw\oplus\frz(\frg)$.
%Let $B_\frw$ denote a basis of $\frw$, $B_\frz$
%a basis of $\frz(\frg)$ and $B_\frw^*,B_\frz^*$ their respective dual bases 
%in $\frg^*$. Then $B=(B^*_\frw,B_\frz,B_\frw,B^*_\frz)$ is a Witt basis
%of $\frg$.
%With respect to $B$, we identify $\frh$ with $\RR^{2n}_n$
%and get a representation
%$\varrho:\frg\to\iso(\RR^{2n}_n)$, where $X=W+Z\in\frg$ with
%$W\in\frw, Z\in\frz(\frg)$ is represented by
%\[
%\varrho(X)=(A_X,v_X)=
%\Bigl(
%\left(\begin{array}{cc|cc}
%0 & 0 & 0 & -\frac{1}{2}\ad_\frw(W)^\top \\
%0 & 0 & \frac{1}{2}\ad_\frw(W) & 0 \\
%\hline
%0 & 0 & 0 & 0 \\
%0 & 0 & 0 & 0
%\end{array}\right),
%\begin{pmatrix}
%0\\ Z\\\hline W\\0
%\end{pmatrix}\Bigr).
%\]
%Then
%\[
%U_\Gamma=\sum_{X\in\frg} \im A_X \subset\mathrm{span}(B^*_\frw,B_\frz),
%\]
%where the space on the right is totally isotropic.
%By Proposition \ref{prop_wolf1}, $\Gamma$ has abelian holonomy.

\section*{Acknowledgments}

I would like to thank Oliver Baues for many helpful remarks.

%%%%% Bibliography %%%%%%

\end{document}